\documentclass[12pt]{article}
\input epsf.tex

\usepackage{amsmath}
\usepackage{amsthm}
\usepackage{amsfonts}
\usepackage{amssymb}
\usepackage{graphicx}
\usepackage{latexsym}

\usepackage{amsmath,amsthm,amsfonts,amssymb,mathrsfs}

\usepackage{epsfig}

\usepackage{amssymb}
\usepackage[all]{xy}
\xyoption{poly}
\usepackage{fancyhdr}
\usepackage{wrapfig}

\theoremstyle{plain}
\newtheorem{thm}{Theorem}[section]
\newtheorem{prop}[thm]{Proposition}
\newtheorem{lem}[thm]{Lemma}
\newtheorem{cor}[thm]{Corollary}

\theoremstyle{definition}
\newtheorem{defn}{Definition}
\theoremstyle{remark}

\advance \topmargin by -\headheight
\advance \topmargin by -\headsep
\evensidemargin \oddsidemargin
\def\cc{{\curvearrowright}}

\def\ce{{\check e}}

\def\F{{\mathbb F}}
\def\Ft{{\mathbb F_2}}
\def\Fk{{\mathbb F_{T}}}
\def\Fkz{{\mathbb F^z_{T}}}
\def\FF{{\mathbb F}}

\def\cG{{\mathcal G}}

\def\len{{\textrm{len}}}
\def\cL{{\mathcal L}}

\def\N{{\mathbb N}}
\def\cN{{\mathcal N}}

\def\rng{{\textrm{rng}}}

\def\cS{{\mathcal S}}

\def\supp{{\textrm{supp}}}

\def\supp{{\textrm{supp}}}

\def\chix{{\raise.5ex\hbox{$\chi$}}}

\def\Z{{\mathbb Z}}

\begin{document}
\title{Stable orbit equivalence of \\Bernoulli shifts over free groups}
\author{Lewis Bowen$^*$}
\thanks{ Supported in part by NSF grant DMS-0901835.}
\begin{abstract}
Previous work showed that every pair of nontrivial Bernoulli shifts over a fixed free group are orbit equivalent. In this paper, we prove that if $G_1,G_2$ are nonabelian free groups of finite rank then every nontrivial Bernoulli shift over $G_1$ is stably orbit equivalent to every nontrivial Bernoulli shift over $G_2$. This answers a question of S. Popa.
\end{abstract}
\maketitle

\noindent
\tableofcontents


\section{Introduction}

Let $G$ be a countable group and $(X,\mu)$ a standard probability space. A probability measure-preserving (p.m.p.) action of $G$ on $(X,\mu)$ is a collection $\{T_g\}_{g\in G}$ of measure-preserving transformations $T_g:X \to X$ such that $T_{g_1}T_{g_2}=T_{g_1g_2}$ for all $g_1,g_2 \in G$. We denote this by $G \curvearrowright^T (X,\mu)$.

Suppose $G_1\cc^{T} (X_1,\mu_1)$ and $G_2\cc^{S} (X_2,\mu_2)$ are two  p.m.p. actions. A measurable map $\phi:X_1' \to X_2'$ (where $X'_i \subset X_i$ is conull) is an {\em orbit equivalence} if the push-forward measure $\phi_*\mu_1$ equals $\mu_2$ and for every $x\in X_1'$, $\{T_gx:~g\in G_1\}=\{S_g\phi(x):~g\in G_2\}$. If there exists such a map, then the actions $T$ and $S$ are said to be {\em orbit equivalent} (OE).

If, in addition, there is a group isomorphism $\Psi:G_1 \to G_2$ such that $\phi(T_g x) = S_{\Psi(g)} \phi(x)$ for every $x\in X'_1$ and $g\in G_1$ then the actions $T$ and $S$ are said to be {\em measurably-conjugate}.

If $A \subset X$ is a set of positive $\mu$-measure then let $\mu(\cdot |A)$ denote the probability measure on $A$ defined by $\mu(E|A) = \frac{\mu(E\cap A)}{\mu(A)}$.  Two p.m.p. actions $G_1 \cc^T (X_1,\mu_1)$ and $G_2 \cc^S (X_2,\mu_2)$ are {\em stably orbit equivalent} (SOE) if there exist positive measure sets $A_i \subset X_i$ and a map $\phi:A_1\to A_2$ inducing a measure-space isomorphism between $(A_1,\mu_1(\cdot|A_1))$ and $(A_2,\mu_2(\cdot|A_2))$ such that for a.e. $x\in A_1$, $\{T_gx:g\in G_1\} \cap A_1= \{S_g\phi(x):g\in G_2\} \cap A_2$.

The initial motivation for orbit equivalence comes from the study of von Neumann algebras. It is known that two p.m.p. actions are orbit equivalent if and only if their associated crossed product von Neumann algebras are isomorphic by an isomorphism that preserves the Cartan subalgebras [Si55].  H. Dye [Dy59, Dy63]  proved the pioneering result that any two ergodic p.m.p. actions of the group of integers on the unit interval are OE. This was extended to amenable groups in [OW80] and [CFW81]. By contrast, it is now known that every nonamenable group admits a continuum of non-orbit equivalent ergodic p.m.p. actions [Ep09]. This followed a series of earlier results that dealt with various important classes of non-amenable groups ([GP05], [Hj05], [Ioxx], [Ki08], [MS06], [Po06]). 

In the last decade, a number of striking OE rigidity results have been proven (for surveys, see [Fu09], [Po07] and [Sh05]). These imply that, under special conditions, OE implies measure-conjugacy. By contrast, the main theorem of this paper could be called an OE ``flexibility'' result. This theorem and those of the related paper [Bo09b] are apparently the first flexibility results in the nonamenable setting.

The new result concerns a special class of dynamical systems called Bernoulli shifts. To define them, let $(K,\kappa)$ be a standard probability space. If $G$ is a countable discrete group, then $K^G$ is the set of all of functions $x:G \to K$ with the product Borel structure. For each $g\in G$, let $S_g:K^G \to K^G$ be the shift-map defined by $S_gx(h):=x(g^{-1}h)$ for any $h\in G$ and $x\in K^G$. This map preserves the product measure $\kappa^G$. The action $G \cc^S (K^G,\kappa^G)$ is called the {\em Bernoulli shift over $G$ with base-space $(K,\kappa)$}. To avoid trivialities, we will assume that $\kappa$ is not supported on a single point.

If $\kappa$ is supported on a finite or countable set $K'\subset K$ then the {\em entropy} of $(K,\kappa)$ is defined by
$$H(K,\kappa) := -\sum_{k \in K'} \kappa(\{k\}) \log\big( \kappa(\{k\})\big).$$
If $\kappa$ is not supported on any countable set then $H(K,\kappa):=+\infty$.

A. N. Kolmogorov proved that if two Bernoulli shifts $\Z \cc (K^\Z,\kappa^\Z)$ and $\Z \cc (L^\Z,\lambda^\Z)$ are measurably-conjugate then the base-space entropies $H(K,\kappa)$ and $H(L,\lambda)$ are equal [Ko58, Ko59]. This answered a question of von Neumann which had been posed at least 20 years prior. The converse to Kolmogorov's theorem was famously proven by D. Ornstein [Or70ab]. Both results were extended to countable infinite amenable groups in [OW87].  

A group $G$ is said to be {\em Ornstein} if whenever $(K,\kappa), (L,\lambda)$ are standard probability spaces with $H(K,\kappa)=H(L,\lambda)$ then the corresponding Bernoulli shifts $G \cc (K^G,\kappa^G)$ and $G \cc (L^G,\lambda^G)$ are measurable conjugate. A. M. Stepin proved that if $G$ contains an Ornstein subgroup, then $G$ is Ornstein [St75]. Therefore, any group $G$ that contains an infinite amenable subgroup is Ornstein. It is not known whether every countably infinite group is Ornstein.

In [Bo09a], I proved that every sofic group satisfies a Kolmogorov-type theorem. Precisely, if $G$ is sofic, $(K,\kappa), (L,\lambda)$ are standard probability spaces with $H(K,\kappa) + H(L,\lambda)<\infty$ and the associated Bernoulli shifts $G \cc (K^G,\kappa^G)$, $G \cc (L^G,\lambda^G)$ are measurably-conjugate then $H(K,\kappa)=H(L,\lambda)$. If $G$ is also Ornstein then the finiteness condition on the entropies can be removed. Sofic groups were defined implicitly by M. Gromov [Gr99] and explicitly by B. Weiss [We00]. For example, every countably infinite linear group is sofic and Ornstein. It is not known whether or not all countable groups are sofic. 

In summary, it is known that for a large class of groups (e.g., all countable linear groups), Bernoulli shifts are completely classified up to measure-conjugacy by base-space entropy. Let us now turn to the question of orbit equivalence.

By aforementioned results of [OW80] and [CFW81], it follows that if $G_1$ and $G_2$ are any two infinite amenable groups then any two nontrivial Bernoulli shifts $G_1 \cc (K^{G_1},\kappa^{G_1})$, $G_2 \cc (L^{G_2},\lambda^{G_2})$ are orbit equivalent. By contrast, it was shown in [Bo09a] that the main result of [Bo09a] combined with rigidity results of S. Popa [Po06, Po08] and Y. Kida [Ki08] proves that for many nonamenable groups $G$, Bernoulli shifts are classified up to orbit equivalence and even stable orbit equivalence by base-space entropy. For example, this includes PSL$_n(\Z)$ for $n>2$, mapping class groups of surfaces (with a few exceptions) and any nonamenable sofic Ornstein group of the form $G=H\times N$ with both $H$ and $N$ countably infinite that has no nontrivial finite normal subgroups.

In [Bo09b] it was shown that if $\FF_r$ denotes the free group of rank $r$ then every pair of nontrivial Bernoulli shifts over $\FF_r$ are OE. By [Ga00], the cost of a Bernoulli shift action of $\FF_r$ equals $r$. Since cost is invariant under OE, it follows that no Bernoulli shift over $\FF_r$ can be OE to a Bernoulli shift over $\FF_s$ if $r \ne s$. Moreover, since SOE preserves cost 1 and cost $\infty$, it follows that no Bernoulli shift over $\FF_1=\Z$ can be SOE to a Bernoulli shift over $\FF_r$ for $r>1$ and no Bernoulli shift over $\FF_\infty$ can be SOE to a Bernoulli shift over $\FF_r$ for finite $r$. The main result of this paper is:

\begin{thm}\label{thm:main}
If $1<r,s < \infty$ then all Bernoulli shift actions over $\FF_r$ and $\FF_s$ are stably orbit equivalent.
\end{thm}
 
 \begin{cor}
 Let $A_1,\ldots,A_r$ and $A'_1,\ldots, A'_s$ be countably infinite amenable groups with $1<s,r<\infty$. Let $\Gamma_1=A_1*\cdots *A_r$ and $\Gamma_2=A'_1*\cdots *A'_s$. Then every Bernoulli shift over $\Gamma_1$ is stably orbit equivalent to every Bernoulli shift over $\Gamma_2$.
 \end{cor}
 \begin{proof}
 From the main result of [Bo09b] it follows that every Bernoulli shift over $\Gamma_1$ is OE to every Bernoulli shift over $\F_r$. Similarly,  every Bernoulli shift over $\Gamma_2$ is OE to every Bernoulli shift over $\F_s$. The result now follows from the theorem above.  \end{proof}

\subsection{Large-scale structure of the proof}

Theorem \ref{thm:main} follows immediately from the two theorems below which will be proven in subsequent sections. To explain them, we need some notation. Let $K$ be a finite or countable set. Then the rank 2 free group $\Ft=\langle a,b\rangle$ acts on $(K\times K)^\Ft$ in the usual way: $(g\cdot x)(f):=x(g^{-1}f)$ for any $g,f\in \Ft$ and $x\in (K\times K)^\Ft$. We call this the {\em shift-action}. Let $\langle b \rangle$ be the cyclic subgroup of $\Ft$ generated by the element  $b$. Define $\Phi: K^{\Ft/\langle b \rangle} \times K^\Ft \to (K\times K)^\Ft$ by 
$$\Phi(x,y)(g)=\Big(x\big(g\langle b \rangle\big), y(g)\Big) ~\forall x \in K^{\Ft/\langle b \rangle}, y\in K^\Ft, g\in \Ft.$$
Observe that $\Phi$ is an injection. So, by abuse of notation, we will identify $K^{\Ft/\langle b \rangle} \times K^\Ft$ with its image $\Phi( K^{\Ft/\langle b \rangle} \times K^\Ft)$. If $\kappa$ is a probability measure on $K$ then let $\kappa^{\Ft/\langle b \rangle} \times \kappa^\Ft$ be the product Borel probability measure on $K^{\Ft/\langle b \rangle} \times K^\Ft$. We extend this measure to all of $(K\times K)^\Ft$ by setting
  $$ \kappa^{\Ft/\langle b \rangle} \times \kappa^\Ft\Big(  (K\times K)^\Ft -   K^{\Ft/\langle b \rangle} \times K^\Ft\Big) =0.$$  
  

\begin{thm}\label{thm:1}
With notation as above, the Bernoulli shift-action $\Ft \cc (K^\Ft,\kappa^\Ft)$ is orbit equivalent to the shift-action $\Ft \cc \Big((K\times K)^\Ft, \kappa^{\Ft/\langle b \rangle} \times \kappa^\Ft\Big)$.
\end{thm}

\begin{thm}\label{thm:2}
Let $K$ be a finite set with  more than one element and let $\kappa$ be the uniform probability measure on $K$. Then the shift-action $\Ft \cc \Big((K\times K)^\Ft, \kappa^{\Ft/\langle b \rangle} \times \kappa^\Ft\Big)$ is SOE to a Bernoulli shift-action of the rank $(|K|+1)$-free group.
\end{thm}

The two theorems above imply that for any $s\ge 3$, there is {\em some} Bernoulli shift over $\F_2$ that is SOE to a Bernoulli shift over $\F_s$. By [Bo09b], we know that all Bernoulli shifts over $\F_r$ are OE for any fixed $r$. So this proves every Bernoulli shift over $\F_2$ is SOE to every Bernoulli shift over $\F_s$. Since SOE is an equivalence relation this proves theorem \ref{thm:main}. Both theorems above are proven by explicit constructions.

 
 \subsection{The main ideas}
 
 In this section we give incomplete non-rigorous proof sketches of the theorems below which serve to illustrate the main ideas of the paper. Let $K=\Z/2\Z$ and let $\kappa$ be the uniform probability measure on $K$.
 
 \begin{thm}\label{thm:a}
 The Bernoulli shift-action $\Ft \cc (K^\Ft,\kappa^\Ft)$ is orbit equivalent to the shift-action $\Ft \cc \Big((K\times K)^\Ft, \kappa^{\Ft/\langle b \rangle} \times \kappa^\Ft\Big)$.
\end{thm}

\begin{thm}\label{thm:b}
The shift-action $\Ft \cc \Big(K^{\Ft/\langle b\rangle}, \kappa^{\Ft/\langle b \rangle}\Big)$ defined by $g\cdot x(C)=x(g^{-1}C)$ for $g\in \Ft$, $x \in K^{\Ft/\langle b\rangle}$ and $C\in \Ft/\langle b\rangle$ is SOE to a nontrivial Bernoulli shift over $\F_3$.
\end{thm}
 
 Theorem \ref{thm:a} follows from theorem \ref{thm:1}. Theorem \ref{thm:b} follows from the proof of theorem  \ref{thm:2}. 
  
 \subsubsection{Proof sketch for theorem \ref{thm:a}}\label{sketch1}
 
  \begin{figure}[htb]
\begin{center}
\ \psfig{file=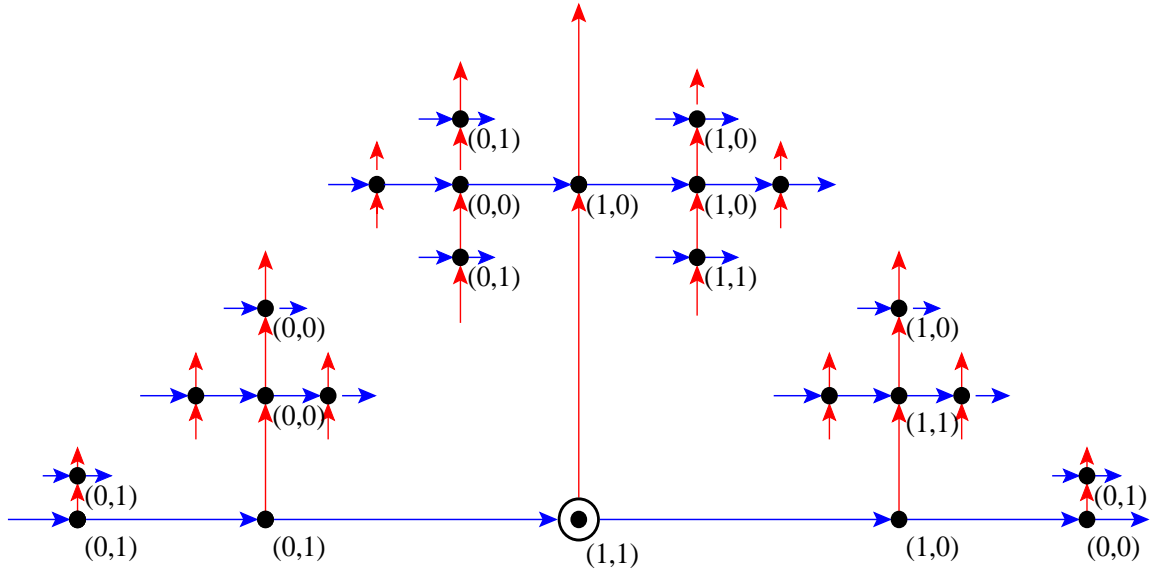,height=3in,width=6in}
\caption{A diagram for a point $x$ in the support of $\kappa^{\Ft/\langle b\rangle} \times \kappa^\Ft$. }
\label{fig:1}
\end{center}
\end{figure}

In figure \ref{fig:1}, there is a diagram for a point $x \in (K\times K)^\Ft$ that is typical with respect to the measure $\kappa^{\Ft/\langle b\rangle} \times \kappa^\Ft$. The underlying graph is the Cayley graph of $\Ft$ (only part of which is shown in the figure). The circled dot represents the identity element in $\Ft$. For every $g \in \Ft$ there are directed edges $(g,ga)$ and $(g,gb)$. Edges of the form $(g,ga)$ are drawn horizontally while those of the form $(g,gb)$ are drawn vertically. Some of the vertices are labeled with an ordered pair which is written to the lower right of the vertex. The ordered pair represents the value of $x$ at the corresponding group element. For example, the diagram indicates that $x(e)=(1,1)$, $x(a)=(1,0)$, $x(a^2)=(0,0)$, $x(b)=(1,0)$ and $x(ba)=(1,0)$. We assume that $x$ is in the support of $\kappa^{\Ft/\langle b\rangle} \times \kappa^\Ft$. So if $x(g)=(i,j)$ for some $g\in \Ft$ and $i,j\in K$ then $x(gb)=(i,k)$ for some $k\in K$. 

To form the orbit equivalence, we will switch certain pairs of $b$-labeled edges. Each switching pair will have their tail vertices on the same coset of $\Ft/\langle a\rangle$. After this is done, and after ``forgetting'' the first coordinates of the labels we will have a diagram of a typical point in $(K^\Ft, \kappa^\Ft)$.

In order to determine which $b$-labeled edges to switch, we place a left square bracket below every vertex labeled $(0,1)$. We place a right square bracket below every vertex labeled $(1,0)$. For example, figure \ref{fig:2} shows part of the Cayley graph with brackets indicated.

\begin{figure}[htb]
\begin{center}
\ \psfig{file=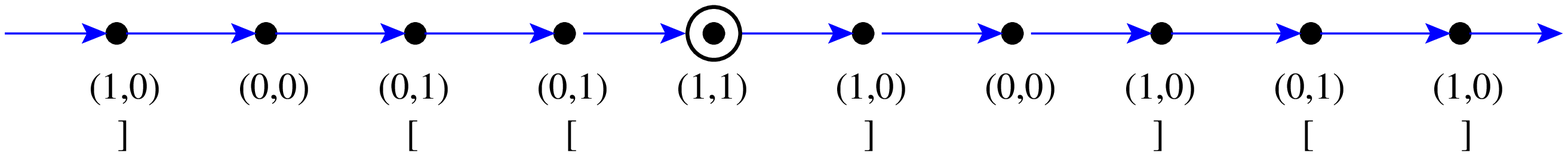,height=0.6in,width=6in}
\caption{}
\label{fig:2}
\end{center}
\end{figure}

The purpose of the brackets is that they give a natural way to pair vertices labeled $(0,1)$ with vertices labeled $(1,0)$. For example, the diagram shows that $a^{-1}$ is to be paired with $a$. Also $a^{-2}$ is paired with $a^3$ and $a^{4}$ is paired with $a^5$. We should emphasize that this occurs all over the group, not just the subgroup $\langle a\rangle$. For example, if $g \in \Ft$ is such that $x(g)=(0,1)$ and $x(ga)=(1,0)$ then $g$ is paired with $ga$.

This pairing of vertices induces a pairing of $b$-labeled edges: two $b$-labeled edges are paired if their source vertices are paired. For example, the diagram tells us that $(a^{-1},a^{-1}b)$ is paired with $(a,ab)$, $(a^{-2}, a^{-2}b)$ is paired with $(a^3,a^3b)$ and so on. The next step is to switch the heads of the paired edges. This is shown in figure \ref{fig:3}.

\begin{figure}[htb]
\begin{center}
\ \psfig{file=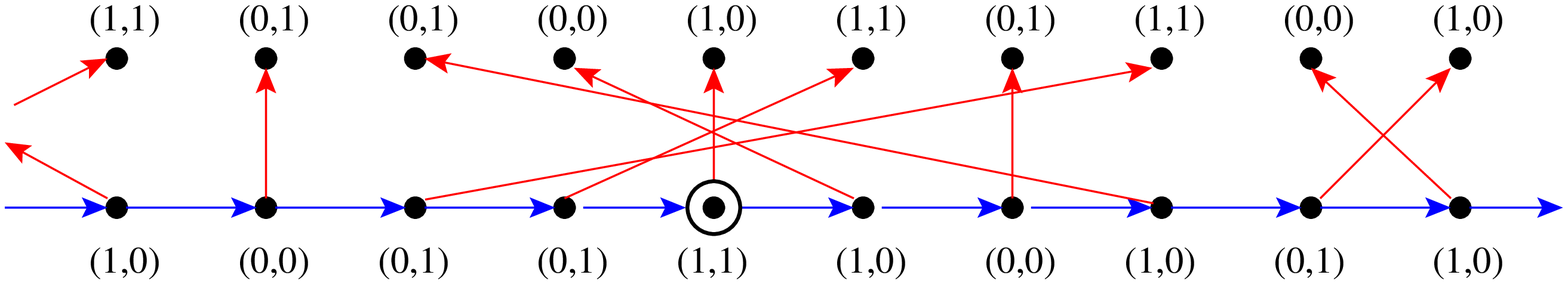,height=1.08in,width=6in}
\caption{}
\label{fig:3}
\end{center}
\end{figure}

After this switching is done, we have a diagram of a point $\Omega x \in (K\times K)^\Ft$. For example, $\Omega x(a^n)=x(a^n)$ for all $n$. According to figure \ref{fig:3}, our example satisfies $\Omega x(ab)= (0,0)$ whereas $x(ab)=(1,1)$. Notice that if $g \in \Ft$ and $\Omega x(g)=(i,j)$ for some $i,j\in K$ then $\Omega x(gb)=(j,k)$ for some $k\in K$. This is because $x$ is in the support of $\kappa^{\Ft/\langle b\rangle} \times \kappa^\Ft$. Therefore if for $i=1,2$, $\pi^\Ft_i:(K\times K)^\Ft \to K^\Ft$ are the projection maps defined by $y(g) = \big( \pi^\Ft_1 y(g) , \pi^\Ft_2 y(g) \big)$ for any $y\in (K\times K)^\Ft$ and $g\in \Ft$, then $\Omega x$ is completely determined by $\pi^\Ft_2\Omega x$. 

We claim that the map $\pi^\Ft_2\Omega: (K\times K)^\Ft \to K^\Ft$ is the required orbit-equivalence. To see this, first observe that $\Omega$ is an involution. Therefore, the map $\pi^\Ft_2 \Omega$ restricted to the support of $\kappa^{\Ft/\langle b\rangle} \times \kappa^\Ft$ is invertible. Because $\Omega$ is defined without mention of the origin, it follows that it takes orbits to orbits. It might not be obvious but $(\pi^\Ft_2\Omega)_*(\kappa^{\Ft/\langle b\rangle} \times \kappa^\Ft)=\kappa^\Ft$. This implies that $\Omega$ is the required orbit-equivalence. The proof of theorem \ref{thm:1} is, in spirit, very much like this sketch.

 \subsubsection{Proof sketch for theorem \ref{thm:b}}\label{sketch2}
 
Let $\Phi:K^{\Ft/\langle b\rangle}\to K^\Ft$ be the map defined by $\Phi(x)(g)=x\big(g\langle b\rangle\big)$. This map is equivariant and injective. So we will identify $K^{\Ft/\langle b\rangle}$ with its image under $\Phi$. We extend the product measure $\kappa^{\Ft/\langle b\rangle}$ to a measure on $K^\Ft$ by setting $\kappa^{\Ft/\langle b\rangle}(K^\Ft - K^{\Ft/\langle b\rangle})=0$. 

 \begin{figure}[htb]
\begin{center}
\ \psfig{file=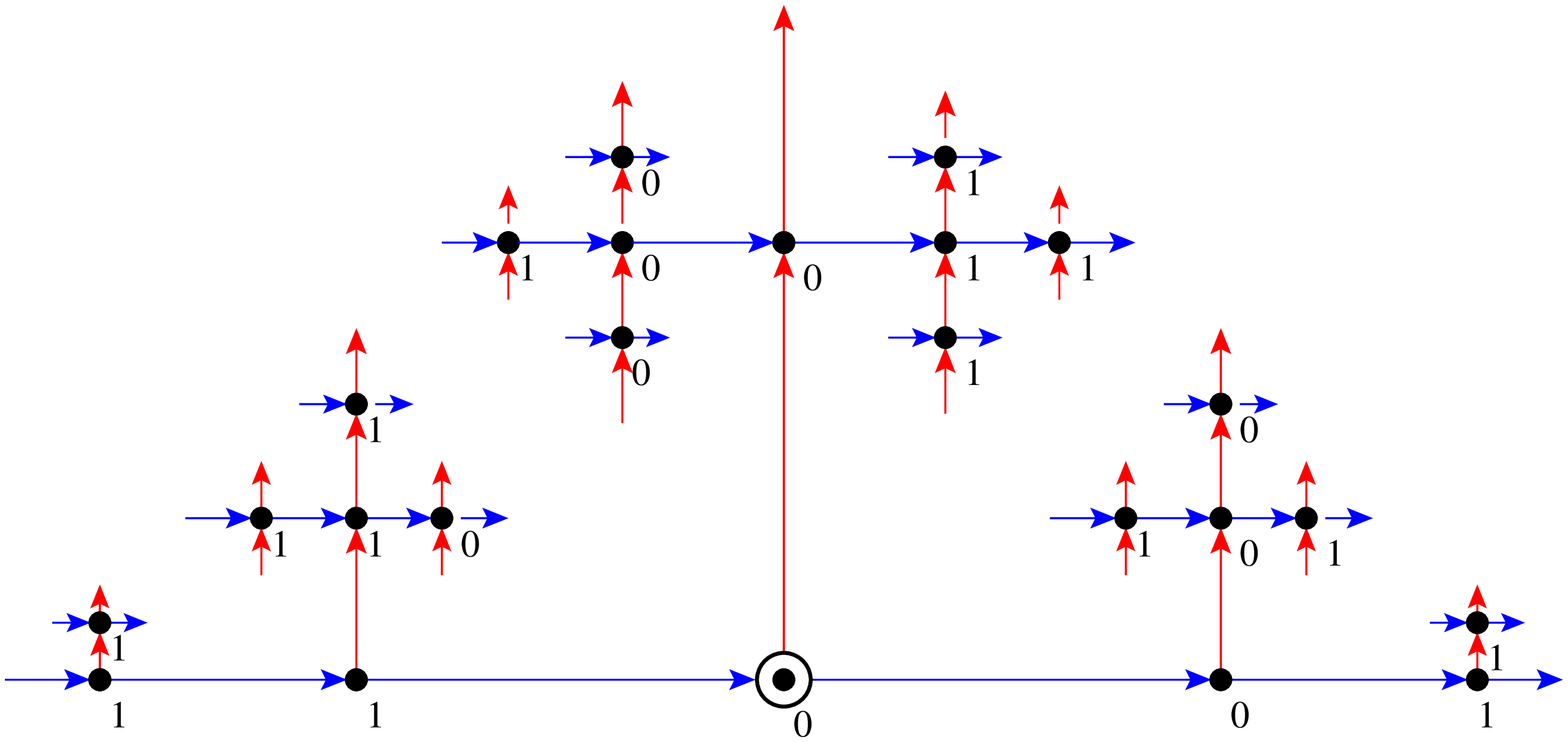,height=3in,width=6in}
\caption{A diagram for a point $x$ in the support of $\kappa^{\Ft/\langle b\rangle}$. }
\label{fig:1b}
\end{center}
\end{figure}

In figure \ref{fig:1b}, there is a diagram for a point $x \in K^\Ft$ that is typical with respect to $\kappa^{\Ft/\langle b\rangle}$.  Each vertex is labeled with a number in $K$ which represents the value of $x$ at the corresponding group element. For example, the diagram indicates that $x(e)=0=x(a)$, $x(a^2)=1$, $x(b)=0$ and $x(ba)=1$. We assume that $x$ is in the support of $\kappa^{\Ft/\langle b\rangle}$. So $x(g)=x(gb)$ for all $g\in \Ft$. 
 
Let us obtain a different diagram for $x$ as follows. Instead of labels on the vertices, we draw the vertical arrows differently: a vertical arrow with both endpoints labeled $1$ is now drawn as a dashed arrow (which is green in the color version of this paper). Vertical arrows with both endpoints labeled $0$ are drawn as before: as solid arrows (which are red in the color version). We also introduce new vertex labels. If $x(g)=0$ then we label the vertex corresponding to $g$ with the smallest positive number $n$ such that $x(ga^n)=0$. We call these {\em distance labels}. The result is shown in figure \ref{fig:2b}.
 
 \begin{figure}[htb]
\begin{center}
\ \psfig{file=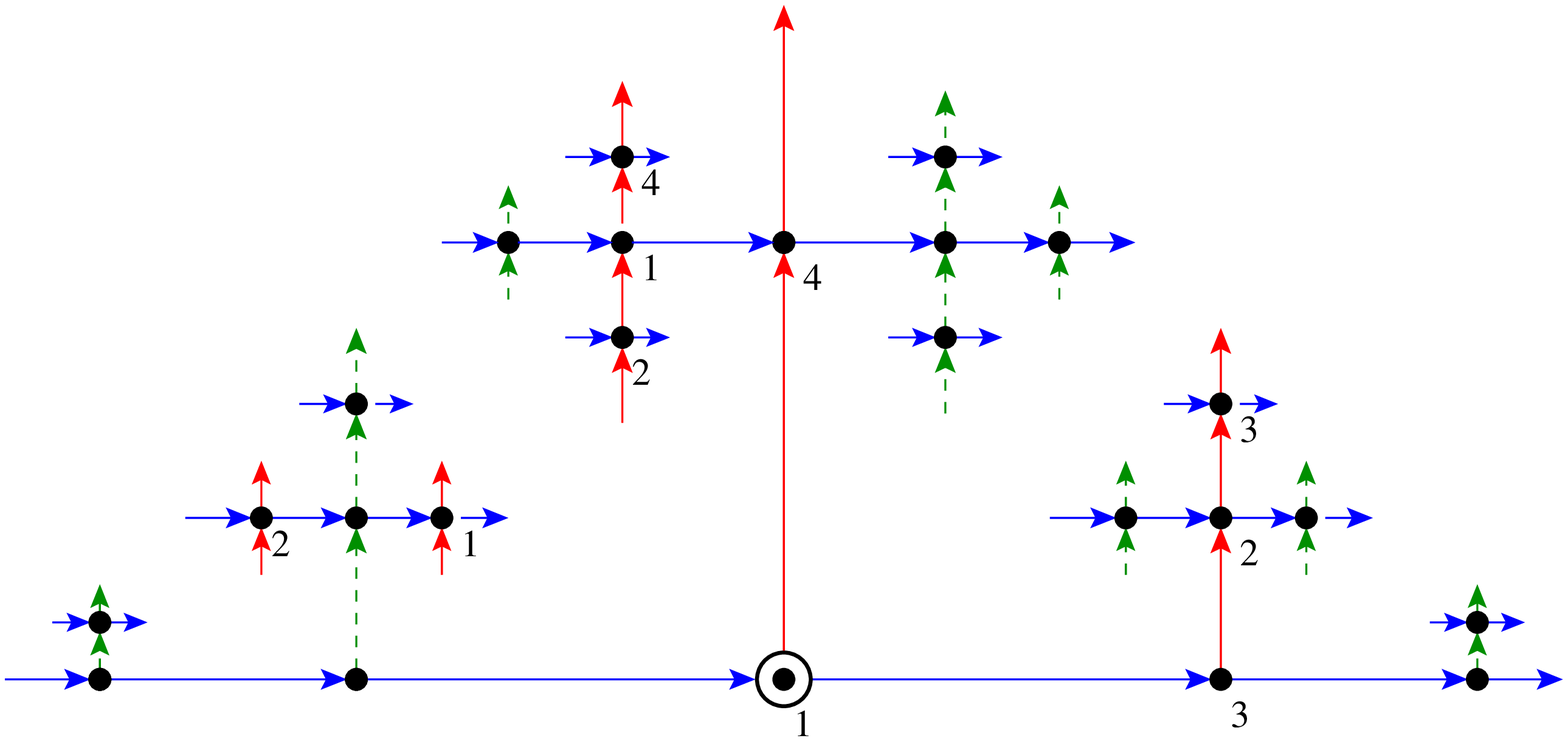,height=3in,width=6in}
\caption{A new diagram for a point $x$ in the support of $\kappa^{\Ft/\langle b\rangle}$. }
\label{fig:2b}
\end{center}
\end{figure}
 
 
 From this diagram for $x$ we will construct a diagram for a point $\Omega x \in \N^{\F_3}$ such that the map $x \mapsto \Omega x$ defines the stable orbit-equivalence. The domain of $\Omega$ will be the set $A_0:=\{y \in K^\Ft:~ y(e)=0\}$. 
 
 We begin by making small changes to the diagram in figure \ref{fig:2b}. First, as in the previous sketch, we place a left bracket next to every vertex that is incident to a solid vertical arrow and a right bracket next to every vertex incident to a dashed vertical arrow. This is shown in figure \ref{fig:3b}. To simplify the picture, we have not written in the distance labels.
 
  \begin{figure}[htb]
\begin{center}
\ \psfig{file=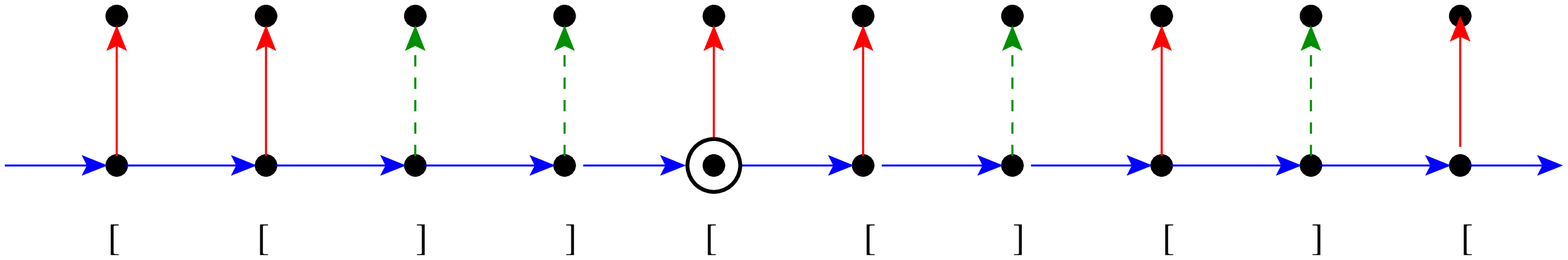,height=0.8in,width=5in}
\caption{}
\label{fig:3b}
\end{center}
\end{figure}
 
 The brackets give a natural way to pair vertices $g$ with $x(g)=0$ with vertices $h$ such that $x(h)=1$. For example, the diagram shows that $a^{-3}$ is paired with $a^{-2}$ and $a^{-4}$ is paired with $a^{-1}$. We should emphasize that this occurs all over the group, not just the subgroup $\langle a\rangle$. For example, if $g \in \Ft$ is such that $x(g)=0$ and $x(ga)=1$ then $g$ is paired with $ga$.
 
 Next, if a vertex $g$ is paired with $ga^n$ for some $n>0$ then we slide the tail of the outgoing dashed vertical arrow incident to $ga^n$ over to $g$. Similarly, we slide the head of the incoming dashed vertical arrow incident to $ga^n$ over to $g$. Figure \ref{fig:4b} shows part of the result of this operation. Note that the heads of the dashed vertical arrows have been moved but for the sake of not complicating the drawing the vertices that they are incident to are not drawn. 

 \begin{figure}[htb]
\begin{center}
\ \psfig{file=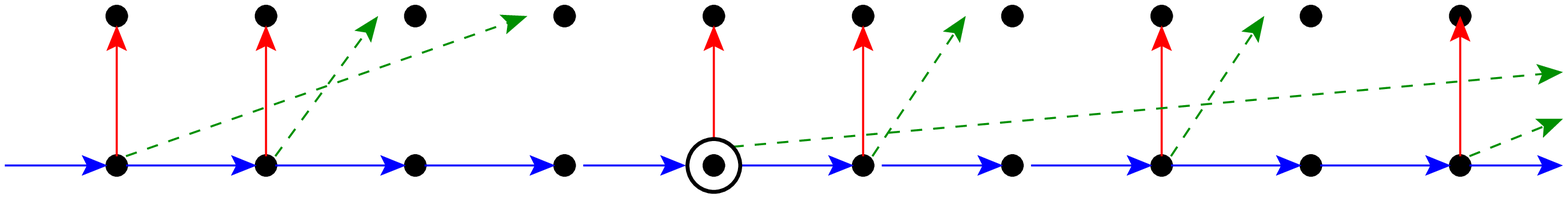,height=0.65in,width=5in}
\caption{}
\label{fig:4b}
\end{center}
\end{figure}

Next, we remove all vertices $g$ with $x(g)=1$. Each one of these vertices is incident to a horizontal arrow coming in and one going out. So when we remove such a vertex, we concatenate these arrows into one. The result is shown in figure \ref{fig:5b}, which also includes the distance labels.

 \begin{figure}[htb]
\begin{center}
\ \psfig{file=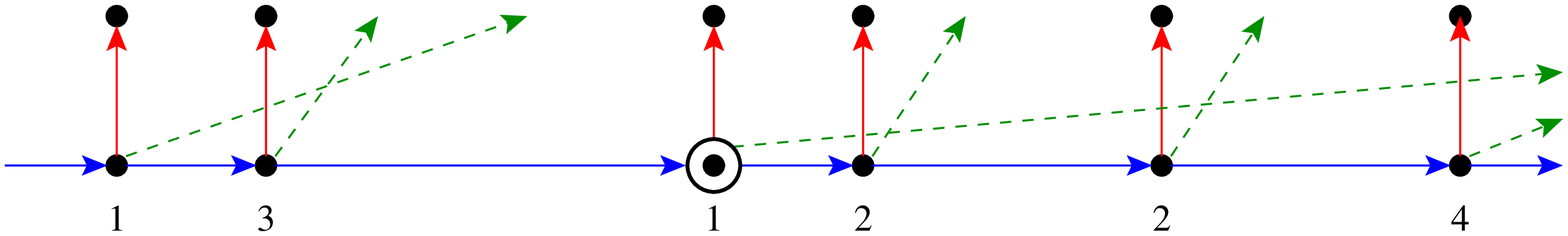,height=0.8in,width=5in}
\caption{}
\label{fig:5b}
\end{center}
\end{figure}

The new diagram is a diagram for a point $\Omega x$ in $\N^{\F_3}$. Here we write $\F_3=\langle a,b,c\rangle$. $a$-edges correspond to horizontal edges, $b$-edges to solid vertical edges and $c$-edges to dashed diagonal edges. For example, the figure above indicates that $\Omega x(e)=1$, $\Omega x(a)=\Omega x(a^2)=2$ and $\Omega x(a^3)=4$.

Now let $\kappa_*$ be the probability measure on $\N=\{0,1,2,3,\ldots\}$ defined by $\kappa_*(\{0\}) = 0$ and $\kappa_*(\{n\}) = 2^{-n}$ for $n\ge 1$. It may not be obvious, but $\Omega$ defines a stable orbit equivalence between $(K^\Ft, \kappa^{\Ft/\langle b\rangle})$ and $(\N^{\F_3},\kappa_*^{\F_3})$. The proof of theorem \ref{thm:2} is based on a very similar construction. 

\subsection{Organization}

In the next section, we discuss rooted networks; how to obtain them from group actions and conversely. In \S \ref{ section: theorem one } and \S \ref{ section: theorem two } we prove theorems \ref{thm:1} and \ref{thm:2} respectively.

{\bf Acknowledgements}. I'd like to thank Sorin Popa for asking whether the full 2-shift over $\F_2$ is SOE to a Bernoulli shift over $\F_3$. My investigations into this question led to this work and the paper [Bo09b].

\section{Rooted networks and orbit equivalence}
The purpose of this section is to introduce rooted networks and discuss their relationships with dynamical systems. They are used as a primary tool in subsequent constructions. 
\subsection{Rooted networks}
\begin{defn}
A rooted network is a quintuple $\cN=(V,E,\cL,\cG,\rho)$ where
\begin{enumerate}
\item $(V,E)$ is a connected and directed graph (so $E \subset V\times V$),
\item $\rho \in V$ is a distinguished vertex called the {\em root},
\item $\cL: V \to \rng(\cL)$ is a map. 
\item $\cG: E \to \rng(\cG)$ is a map.
\end{enumerate}
\end{defn}
$\cL$ and $\cG$ are called the vertex and edge labels respectively. Throughout this paper, $\rng(\cL)$ and $\rng(\cG)$ are finite or countable discrete sets. $\rng(\cG)$ will typically be a set of generators for a free group.

There is a natural Borel structure on the space of all rooted networks [AL07]. To define it, we need the following.
\begin{defn}
Two rooted networks are isomorphic if there an isomorphism of the underlying directed graphs that takes the root of the source network to the root of the target network and preserves both vertex and edge labels. 
\end{defn}

For $n\ge 0$, let $B_n(\cN)$ be the ball of radius $n$ centered at the root of $\Gamma$. It is itself a rooted network with the restricted vertex and edge labels. Define the distance between two rooted networks $\cN_1, \cN_2$ to be $\frac{1}{n+1}$ where $n\ge 0$ is the largest number such that $B_n(\cN_1)$ is isomorphic to $B_n(\cN_2)$ as rooted networks. If no such number exists, then let the distance between them equal $2$. This makes the space of all isomorphism classes of rooted networks with vertex degrees bounded by some number $d >0$ into a compact metric space. We will only use the Borel structure that this induces. See [AL07] for more background on rooted networks (but be warned: the definition here differs somewhat from the definition in [AL07]).

\subsection{Rooted networks from group actions}\label{ section: rooted networks from group actions }

Let $G$ be a discrete countable group and let $x\in K^G$ where $K$ is a countable or finite set. Let $\cS=\{s_1,\ldots,s_r\} \subset G$ generate $G$ as a group. The rooted network $\cN_x=(V_x,E_x,\cL_x,\cG_x,\rho_x)$ induced by $x$ and $\cS$ is defined by: $V_x=G$, $E_x=\{(g,gs):~g\in G, s\in \cS\}$, $\cL_x:G \to K$ satisfies $\cL_x(g)=x(g)$ and $\cG_x:E \to \cS$ satisfies $\cG_x(g,gs)=s$. The root $\rho_x$ is the identity element in $G$.

Recall that $G$ acts on $K^G$ by $(g\cdot x)(f)=x(g^{-1}f)$ for $g,f\in F$ and $x\in K^G$.
\begin{lem}\label{ lemma: root change }
Let $x\in K^G$ and $\cN_x=(V_x,E_x,\cL_x,\cG_x,\rho_x)$. For any $g\in G$ and $x \in K^G$, $\cN_{g\cdot x}$ is isomorphic to $(V_x,E_x,\cL_x,\cG_x,g^{-1})$. 
\end{lem}
\begin{proof}
This is an easy exercise left to the reader.
\end{proof}

\subsection{Group actions from rooted networks}

\begin{defn}
Let $\cN=(V,E,\cL,\cG,\rho)$ be a rooted network. Let $\cS=\rng(\cG)$ be the set of edge labels of $\cN$. We will say that $\cN$ is {\em actionable} if for each $v\in V$ and each $s \in \cS$ there is a unique edge $e\in E$ such that $v$ is the source of $e$ and $\cG(e)=s$. We also require the  existence of a unique edge $\ce \in E$ such that $v$ is the range of $\ce$ and $\cG(\ce)=s$.
\end{defn}

If $\cN$ is actionable then we define an action of $\F_\cS$, the free group with generating set $\cS$, on $V$ by $v\cdot s = w$ where $w$ is the  range of the edge $e$ (as defined above). Also $v\cdot s^{-1}= u$ where $u$ is the source of the edge $\ce$. Observe that this is a right-action of $\F_\cS$ on $V$. For any $g\in \F_\cS$ and $v\in V$, let $A(\cN:v,g)=v\cdot g$. 
 

\subsection{Orbit morphisms from rooted networks}\label{ section: orbit morphisms }

\begin{defn}
Let $G_1 \cc^T (X_1,\mu_1)$ and $G_2 \cc^S (X_2,\mu_2)$ be two dynamical systems. An {\em orbit morphism} from the first  system to the second is a measurable map $\phi:X'_1 \to X_2$ such that for all $g\in G_1$ and $x \in X'_1$ there exists an $h\in G_2$ such that $\phi(T_gx)=S_h\phi(x)$. Here $X'_1 \subset X_1$ is a conull set.
\end{defn}

Let $\cS$ be a set, $\F=\F_\cS$ and let $\mu$ be a shift-invariant Borel probability measure on $K^{\F}$. Let $X \subset K^\F$ be a shift-invariant Borel set with $\mu(X)=1$. Suppose that for each $x\in X$ there is a map $\phi_x:E \to V\times V$ (where $\cN_x=(V_x,E_x,\cL_x,\cG_x,\rho_x)$ is the rooted network induced by $x$ and $\cS$). We may identify $V_x$ with $\F$ and $E_x$ with $\F \times \cS$ by the map $(g,gs) \mapsto (g,s)$. Thus $\phi_x$ can be thought of as a point in the space of all maps from $\F\times \cS$ to $\F \times \F$ which we endow with the topology of uniform convergence on finite subsets. Suppose that  the following hold.
\begin{enumerate}
\item The map $x\mapsto \phi_x$ is Borel. 
\item $ \big(V_x,\phi_x(E_x) \big)$ is connected.
\item For any $x\in X$ and $g\in \F$, $\phi_{g\cdot x} = g^{-1}\phi_x g$ where $g (v,w)=(g v, gw)$ for any edge $(v,w) \in E_x$ and any $g\in \F$. Here we are considering $v$ and $w$ as elements of $\F=V_x$ so that the multiplication $gv$ is in $\F$.

\item $\phi_x$ is injective and if $\cG^\phi_x:\phi_x(E)\to \cS$ is defined by $\cG^\phi_x \big(\phi_x(e) \big)=\cG_x(e)$ then the network $\cN^\phi_x:= \big(V_x,\phi_x(E_x),\cL_x,\cG^\phi_x,\rho_x \big)$ is actionable.
\end{enumerate}
Then define $\Omega:X \to K^\F$ by $\Omega x(g) = x \big( A(\cN^\phi_x: \rho_x, g) \big)$.

\begin{lem}\label{ lemma: orbit morphisms }
For any $x\in X$, the rooted network $\cN_{\Omega x}$ induced by $\Omega x$ and $\cS$ is isomorphic to $\cN^\phi_x$. Moreover, $\Omega$ is an orbit morphism. \end{lem}

\begin{proof}
 The first statement is an easy exercise left to the reader. The third item above implies that for any $g\in \F$, $\cN_{\Omega (g \cdot x)}$  is isomorphic to  $\big(V_x,\phi_x(E_x),\cL_x,\cG^\phi_x,g^{-1} \big)$.  By  the previous lemma, the latter is isomorphic to $\cN_{h \cdot \Omega x}$  for some $h\in \F$. This implies $\Omega(g\cdot x) = h\cdot \Omega x$. So $\Omega$ is an orbit morphism.
\end{proof}

\section{Theorem \ref{thm:1}}\label{ section: theorem one }

\subsection{The pairing}
To begin the proof of theorem \ref{thm:1}, we define a map that will play the role of the brackets of the sketch in \S \ref{sketch1}. Without  loss of generality we may assume $K=\N$. For $x \in (\N\times \N)^\Ft$ and $g\in \Ft$ define $P(x,g) \in \Ft$ as follows.
\begin{itemize}
\item If $x(g)=(i,i)$ for some $i\in \N$ then $P(x,g)=g$.
\item If $x(g)=(i,j)$ for some $i<j$ then let $P(x,g)=ga^n$ where $n>0$ is the smallest number such that 
\begin{enumerate}
\item $x(ga^n)=(j,i)$,
\item $\big|\{m \in \Z \cap [0,n]:~ x(ga^m)=(i,j)\}\big| = \big|\{m \in \Z \cap [0,n]:~ x(ga^m)=(j,i)\}\big|$.
\end{enumerate}
\item  If $x(g)=(j,i)$ for some $i<j$ then let $P(x,g)=ga^{-n}$ where $n>0$ is the smallest number such that 
\begin{enumerate}
\item $x(ga^{-n})=(i,j)$,
\item $\big|\{m \in \Z\cap [0,n]:~ x(ga^{-m})=(i,j)\}\big| = \big|\{m \in \Z\cap [0,n]:~ x(ga^{-m})=(j,i)\}\big|$.
\end{enumerate}
\end{itemize}
A-priori, $P(x,g)$ may not be well-defined since there might not exist a number $n$ satisfying the above conditions. However, we have:
\begin{lem}
Let $X$ be the set of all $x \in (\N \times \N)^\Ft$ such that for all $g\in \Ft$, $P(x,g)$ is well-defined. Then $\kappa^{\Ft/\langle b\rangle} \times \kappa^\Ft(X)=1$. Moreover, $P\big(x,P(x,g) \big)=g$  for all $x\in X, g\in \Ft$.
\end{lem}
\begin{proof}
This is an easy exercise left to the reader. Indeed, if $\mu$ is any shift-invariant Borel probability measure on $(\N\times \N)^\Ft$ such that $\mu\big(\{x :~ x(e)=(i,j)\}\big) = \mu\big(\{x :~ x(e)=(j,i)\}\big)$ for all $i,j \in \N$ then $\mu(X)=1$.

\end{proof}

\subsection{An orbit  equivalence}

In this section, we define a map that plays the role of the switching in the sketch of \S \ref{sketch1}.

Recall that $\Ft=\langle a,b \rangle$. Let $x\in X$. Let $\cN_x=(V_x,E_x,\cL_x,\cG_x,\rho_x)$ be the rooted network induced by $x$ and $\cS=\{a,b\}$. For each edge $e \in E$ define $\phi_x(e)$ by:
\begin{enumerate}
\item if $e=(g,ga)$ for some $g\in \Ft=V_x$ then $\phi_x(e):=e$,
\item if $e=(g,gb)$ for some $g\in \Ft=V_x$ then $\phi_x(e) :=  \big(g, P(x,g)b \big)$.
\end{enumerate}

\begin{lem} The map $x\mapsto \phi_x$ is Borel and for  any $x\in X$  the following hold.
\begin{enumerate}
\item $ \big(V_x,\phi_x(E_x) \big)$ is connected.
\item For any $x\in X$ and $g\in \Ft$, $\phi_{g\cdot x} = g^{-1}\phi_x g$.

\item $\phi_x$ is injective and if $\cG^\phi_x:\phi_x(E)\to \cS$ is defined by $\cG^\phi_x \big(\phi_x(e) \big)=\cG_x(e)$ then the network $\cN^\phi_x:= \big(V_x,\phi_x(E),\cL_x,\cG^\phi_x,\rho_x \big)$ is actionable.
\end{enumerate}
\end{lem}

\begin{proof}
This is an easy exercise left to the reader.
\end{proof}

Define $\Omega:X \to (\N \times \N)^\Ft$ as in \S \ref{ section: orbit morphisms }. I.e., $\Omega x(g) := x \big( A(\cN^\phi_x: \rho_x, g) \big)$ for all $g\in \Ft$.

\begin{lem}\label{ lemma: orbit equivalence }
$\Omega(X) \subset X$. Moreover, $\Omega(\Omega x)=x$ for any $x\in X$. Thus $\Omega$ is an orbit-equivalence from the shift-action $\Ft \cc \big((\N\times \N)^\Ft, \kappa^{\Ft/\langle b\rangle} \times \kappa^\Ft\big)$ to the shift-action $\Ft \cc \big((\N\times \N)^\Ft, \mu\big)$ where $\mu=\Omega_*  \big(\kappa^{\Ft/\langle b\rangle} \times \kappa^\Ft\big)$.
\end{lem}

\begin{proof}
$\Omega$ is an orbit morphism by lemma \ref{ lemma: orbit morphisms }. That $\Omega(\Omega x)=x$ follows from the fact that $P \big(x,P(x,g) \big)=g$ for any $x\in X$ and $g\in \Ft$.
\end{proof}

\subsection{A measure space isomorphism}\label{ section: isomorphism }

In  this section, we prove $\Ft \cc \big((\N\times \N)^\Ft, \mu\big)$  is measurably conjugate  to $\Ft \cc \big(\N^\Ft, \kappa^\Ft\big)$. We will need the next lemma.

For $x\in X$, define $\cN^\phi_x = \big(V_x,\phi_x(E_x),\cL_x,\cG^\phi_x,\rho_x \big)$ as above. For $g\in \Ft$, define   $\alpha_x(g):=A(\cN^\phi_x:\rho_x,g)$. So $\Omega x(g)=x \big(\alpha_x(g) \big)$. 

\begin{lem}\label{lem:alpha}
For $x\in X$, $g \in \Ft$ and $n\in \Z$,
\begin{eqnarray*}
\alpha_x(ga^n)&=&\alpha_x(g)a^n\\
\alpha_x(gb)&=&P\big(x,\alpha_x(g) \big)b\\
\alpha_x(gb^{-1})&=& P \big(x,\alpha_x(g)b^{-1} \big).
\end{eqnarray*}


 \end{lem}
\begin{proof}
This is an easy exercise  in understanding the definitions.
\end{proof}

 \begin{defn}
  Let $K$  be a set. For $x \in (K \times K)^\Ft$  define $x_1,x_2 \in K^\Ft$ by
 $$x(g):=\big(x_1(g),x_2(g)\big) ~\forall g\in \Ft.$$
  Also   for $i=1,2$, define $\pi_i^\Ft: (K \times K)^\Ft \to K^\Ft$  by $\pi_i^\Ft(x)=x_i$.
  \end{defn}

\begin{lem}\label{ lemma: coordinates }
Let $x\in X$ be in the support of $\kappa^{\Ft/\langle b\rangle} \times \kappa^\Ft$. Then for any $g\in \Ft$, $(\Omega x)_1(gb)=(\Omega x)_2(g)$.
\end{lem}

\begin{proof}
Since $x$ is in the support of $\kappa^{\Ft/\langle b\rangle} \times \kappa^\Ft$, for any $g\in \Ft$, $x_1(g)=x_1(gb)$. 


 Now fix $g \in \Ft$ and let $f \in \Ft$ be such that $A(\cN^\phi_x: \rho_x,g)=f \in \Ft$. By  the previous lemma, $A(\cN^\phi_x: \rho_x,gb) = P(x,f)b$. Thus $\Omega x(g)=x(f)$ and $\Omega x(gb) = x \big(P(x,f)b \big)$. 
 

If $x(f)=(i,j)$ then by definition of $P$, $x \big(P(x,f) \big)=(j,i)$. Since $x$ is in the support of  $\kappa^{\Ft/\langle b\rangle} \times \kappa^\Ft$, $x \big(P(x,f)b \big)=(j,k)$ for some $k$. So $(\Omega x)_2(g)=x_2(f)=j=x_1\big(P(x,f)b\big)=(\Omega x)_1(gb)$. This proves the lemma.
\end{proof}

\begin{defn}
The right-Cayley graph $\Gamma$ of $\Ft$ is the graph with vertex set $\Ft$ and edges $\{(g,gs):~s\in \cS\}$. If $W \subset \Ft$ then the induced subgraph of $W$ is the largest subgraph of $\Gamma$ with vertex set $W$. If this subgraph is connected then we say that $W$ is {\em right-connected}.
\end{defn}


Given a measurable function $f:X\to Y$, the $\sigma$-algebra that it induces on $X$, denoted $\Sigma(f)$, is the pullback $f^{-1}(\Sigma_Y)$ where $\Sigma_Y$ is the $\sigma$-algebra on $Y$.  We will say that a function  $f_1$ is {\em determined by} a function $f_2$ if the sigma algebra induced by $f_1$ is contained in the sigma algebra induced by $f_2$  up to sets of $\kappa^{\Ft/\langle b\rangle} \times \kappa^\Ft$ measure zero. 

Often it will be that we have to consider a function $f(x,i)$ that depends on two arguments $x$ and $i$. This can be considered as a function of $x$ with range a function of $i$. Thus we will write $x \mapsto \big[f(x,i)\big]_{i \in I}$ to mean $x \mapsto \Big( i \in I \mapsto f(x,i) \Big)$. We will also write this as $x\mapsto \big[ f(x,i):~i\in I\big]$.






 \begin{lem}
Suppose that $W \subset \Ft$ is a right-connected set such that $Wa = W$ and $e\in W$. Then for any $v\in W$, the function $x\mapsto \alpha_x(v)$ is determined by the function $x\mapsto \big(\Omega x(w)\big)_{w\in W}$. Similarly, $x\mapsto P \big(x,\alpha_x(v) \big)$ is determined by  $x\mapsto \big((\Omega x)(w)\big)_{w\in W}$.
\end{lem}

\begin{proof}
By definition of $P$, for any fixed $v\in W$, $x\mapsto \alpha_x(v)^{-1}P \big(x,\alpha_x(v) \big)$ is determined by the function $x \mapsto \big[x \big(\alpha_x(v)a^n \big)\big]_{n \in \Z} = \big[\Omega x(va^n)\big]_{n\in \Z}$. Thus the second statement follows from the first.


Since, for any fixed $n\in \Z$, $\alpha_x(a^n)=a^n$ and $x\mapsto P(x,a^n)$ is determined by $x\mapsto \{x(a^m):~m \in \Z\}$, the lemma is true if $W=\{a^n:~n\in \Z\}$. Suppose, for induction, that the lemma is true for a given set right-connected set $W$ with $e\in W$ and $Wa=W$. Let $g\in W$. It suffices to prove that the lemma is true for $W \cup gb\langle a\rangle$ and $W \cup gb^{-1}\langle a\rangle$.

By lemma \ref{lem:alpha}, $\alpha_x(gba^n)=P \big(x,\alpha_x(g) \big)ba^n$ (for any $x\in X, n \in Z$). By induction, $x \mapsto P \big(x,\alpha_x(g) \big)$ is determined  by $x\mapsto \big[(\Omega x)(w)\big]_{w\in W}$. So for any $n$, $x\mapsto \alpha_x(gba^n)$ is determined  $x\mapsto \big[(\Omega x)(w)\big]_{w\in W}$. This proves that the lemma is true for $W \cup gb\langle a\rangle$.

By lemma \ref{lem:alpha}, $\alpha_x(gb^{-1})= P \big(x,\alpha_x(g)b^{-1} \big)$ (for any $x\in X$). Since $P(x,P(x,f))=f$ for any $f$, $P \big(x, \alpha_x(gb^{-1}) \big) = \alpha_x(g)b^{-1}$. So,
\begin{eqnarray*}\label{eqn:a}
\big[ \alpha_x(g)b^{-1}  \big]^{-1} P \big(x,\alpha_x(g)b^{-1} \big) = \Big[ \alpha_x(gb^{-1})^{-1} P \big(x, \alpha_x(gb^{-1}) \big)  \Big]^{-1}. 
\end{eqnarray*}

Since $x\mapsto  \alpha_x(gb^{-1})^{-1} P \big(x, \alpha_x(gb^{-1}) \big) $ is determined by $x\mapsto \big[\big(\alpha_x(gb^{-1})a^n \big)\big]_{n\in \Z} = \big[\Omega x(gb^{-1}a^n)\big]_{n\in \Z}$, it follows that $ x\mapsto \big[ \alpha_x(g)b^{-1}  \big]^{-1} P \big(x,\alpha_x(g)b^{-1} \big)$  is determined  by $x\mapsto \big[(\Omega x)(gb^{-1}a^n):~n\in \Z\big]$. The induction hypothesis implies that $x \mapsto \alpha_x(g)$ is determined  by $x\mapsto \big[(\Omega x)(w)\big]_{w\in W}$ from which it now follows that $x\mapsto P \big(x,\alpha_x(g)b^{-1} \big)=\alpha_x(gb^{-1})$ is determined by $x\mapsto \big[(\Omega x)(w):~w\in W\cup gb^{-1}\langle a\rangle \big]$. Since $\alpha_x(gb^{-1}a^n)=\alpha_x(gb^{-1})a^n$, the lemma is true for $W \cup gb^{-1}\langle a\rangle$. This completes the induction step and hence the lemma.







\end{proof}

\begin{prop}
 $(\pi_2^\Ft \Omega)_*\big( \kappa^{\Ft/\langle b\rangle} \times \kappa^\Ft\big) = \kappa^\Ft$.
 \end{prop}
 
 \begin{proof}
 
  Let $x\in X$ be a random variable with law $\kappa^{\Ft/\langle b\rangle} \times \kappa^\Ft$.  By shift-invariance, it suffices to show that $\{(\Omega x)_2(g)\}_{g\in \Ft}$ is a collection of i.i.d. random variables. 

 Since $\Omega x(a^n)=x(a^n)$ for any $n \in \Z$, the variables $\{(\Omega x)_2(a^n)\}_{n \in \Z}$ are independent identically distributed (i.i.d.) each with law $\kappa$. Suppose, for induction, that $W \subset \Ft$ is a right-connected set such that $Wa = W$, $e\in W$ and $\{(\Omega x)_2(w)\}_{w \in W}$ is an i.i.d. collection. We will show that for any $g\in W$:
\begin{enumerate}
\item if $V=W \cup gb\langle a\rangle$ then $\{(\Omega x)_2(v)\}_{v \in V}$ is an i.i.d. collection; 
 \item if $V=W \cup gb^{-1}\langle a\rangle$ then $\{(\Omega x)_2(v)\}_{v \in V}$ is an i.i.d. collection. 
\end{enumerate} 
By induction, this will prove the proposition. 

Recall that two  measurable functions $f_1, f_2$  with domain $X$  are  independent if for every pair of sets $A_1,A_2$ such that $A_i$ is in the  sigma algebra induced by $f_i$ ($i=1,2$), $$ \big(\kappa^{\Ft/\langle b\rangle} \times \kappa^\Ft \big)(A_1 \cap A_2) =  \big(\kappa^{\Ft/\langle b\rangle} \times \kappa^\Ft\big)(A_1) \big(\kappa^{\Ft/\langle b\rangle} \times \kappa^\Ft\big)(A_2).$$

To prove item (1.), we may assume that $gb \notin W$ since otherwise $W=V$ and item (1.) is trivial. By the previous lemma, $x\mapsto \alpha_x(g)$ and $x\mapsto P \big(x,\alpha_x(g) \big)$ are determined by $x\mapsto \big[(\Omega x)(w)\big]_{w\in W}$. The function 
$$x\mapsto \Big[x_2 \Big( P \big(x,\alpha_x(g) \big)ba^n \Big)\Big]_{n\in \Z}=\big[(\Omega x)_2(gba^n)\big]_{n\in \Z}$$
 is independent of $\big[(\Omega x)(w)\big]_{w\in W}$ because the set $ \Big\{P \big(x,\alpha_x(g) \big)ba^n:~n\in \Z \Big\}$ is disjoint from the set $\{\alpha_x(w):~ w \in W\}$ and $x\mapsto P \big(x,\alpha_x(g) \big)$ is determined by $ x\mapsto \big[x \big(\alpha_x(w) \big) \big]_{w\in W}$. The induction hypothesis now implies item (1.).

To prove item (2.), we may assume that $gb^{-1} \notin W$ since otherwise $W=V$ and item (2.) is trivial. Note 
 \begin{eqnarray*}
P \big(x,\alpha_x(g)b^{-1} \big) &=& \alpha_x(g)b^{-1} [\alpha_x(g)b^{-1} ]^{-1} P \big(x,\alpha_x(g)b^{-1} \big).
\end{eqnarray*}
By the previous lemma, $x\mapsto \alpha_x(g)$ is determined by $x\mapsto \big[(\Omega x)(w)\big]_{w\in W}$. By definition of $P$, 
$$[\alpha_x(g)b^{-1} ]^{-1} P \big(x,\alpha_x(g)b^{-1} \big)$$
is determined by the function $x\mapsto  \big[\alpha_x(g)b^{-1}a^n \big]_{n\in \Z} $ which is independent of $x\mapsto \big[(\Omega x)_2(w)\big]_{w\in W}$ since  $ \Big\{\alpha_x(g)b^{-1}a^n:~n\in \Z \Big\}$ is disjoint from $\{\alpha_x(w):~ w \in W\}$. Therefore the function 
$$x \mapsto x_2 \Big[ P \big(x,\alpha_x(g) b^{-1}\big)a^n \Big]_{n\in \Z}=\big[(\Omega x)_2(gb^{-1}a^n)\big]_{n\in\Z}$$
 is independent of $x\mapsto \big[(\Omega x)(w)\big]_{w\in W}$. This uses again the fact that $\Big\{P \big(x,\alpha_x(g)b^{-1} \big)a^n:~n\in \Z \Big\}$ is disjoint from the set $\{\alpha_x(w):~ w \in W\}$. The induction hypothesis now implies item (2.).


 \end{proof}

We can now prove theorem \ref{thm:1}.

\begin{proof}[Proof of theorem \ref{thm:1}]
It follows from lemma \ref{ lemma: coordinates } that the map $\pi_2^\Ft: \supp(\Omega_*\kappa^{\Ft/\langle b\rangle} \times \kappa^\Ft) \to \N^\Ft$ is invertible. It follows from the previous proposition that this map is a measure-conjugacy from the shift-action $\Ft \cc \big( (\N \times \N)^\Ft, \Omega_* (\kappa^{\Ft/\langle b\rangle} \times \kappa^\Ft) \big)$ to the shift-action  $\Ft \cc (\N^\Ft, \kappa^\Ft)$.  By  lemma \ref{ lemma: orbit equivalence }, $\Omega$ is an OE from $\Ft \cc  \big( (\N \times \N)^\Ft, \kappa^{\Ft/\langle b\rangle} \times \kappa^\Ft \big)$ to
  $\Ft \cc  \big( (\N \times \N)^\Ft, \Omega_* (\kappa^{\Ft/\langle b\rangle} \times \kappa^\Ft) \big)$.
\end{proof}  

\section{Theorem \ref{thm:2}}\label{ section: theorem two }
As in the statement of theorem \ref{thm:2}, let $K$ be a finite set with $|K|\ge 2$. We will assume that there are two elements $0,1$ such that $1 \in K$ but $0 \notin K$. These elements will later be related to the generators $\{a,b\}$ of $\Ft$.

\subsection{The pairings}
To begin the proof of theorem \ref{thm:2}, we define a set of maps that will play the role of the brackets of the sketch in \S \ref{sketch2}. For $x \in (K\times K)^\Ft$, define $x_1 \in K^\Ft$ and $x_2 \in K^\Ft$ as in the previous  section. So, $x(g)=\big(x_1(g),x_2(g)\big)$ for any $g\in\Ft$. For $g\in \Ft$ and $k\in K$ define $P_k(x,g) \in \Ft$ as follows.
\begin{itemize}
\item If $k=1$, then $P_k(x,g):=g$.
\item If $x_1(g)\notin  \{1,k\}$ then $P_k(x,g):=g$.
\item If $k\ne 1$ and $x_1(g)=1$ then let $P_k(x,g)=ga^n$ where $n>0$ is the smallest number such that 
\begin{enumerate}
\item $x_1(ga^n)=k$,
\item $ \big|\{m \in \Z \cap [0,n]:~ x_1(ga^m)=1\}\big| =  \big|\{m \in \Z\cap [0,n]:~ x_1(ga^m)=k\} \big|$.
\end{enumerate}
\item  If $k\ne 1$ and $x_1(g)=k$ then let $P_k(x,g)=ga^{-n}$ where $n>0$ is the smallest number such that 
\begin{enumerate}
\item $x_1(ga^{-n})=1$,
\item $ \big|\{m \in \Z\cap [0,n]:~ x_1(ga^{-m})=1\} \big| =  \big|\{m \in \Z\cap [0,n]:~ x_1(ga^{-m})=k\} \big|$.
\end{enumerate}
\end{itemize}
A-priori, $P_k(x,g)$ may not be well-defined since there might not exist a number $n$ satisfying the above conditions. However, we have:
\begin{lem}
Let $X$ be the set of all $x \in (K\times K)^\Ft$ such that for all $g\in \Ft$ and all $k\in K$ $P_k(x,g)$ is well-defined. Then $\kappa^{\Ft/\langle b\rangle} \times \kappa^\Ft(X)=1$. Moreover, $P_k \big(x,P_k(x,g) \big)=g$  for any $x\in X$  and $g\in \Ft$.
\end{lem}
\begin{proof}
This is an easy exercise left to the reader.
\end{proof}
In this section, $X$ will denote the set defined above. It is not the same as the set $X$ defined in the previous section of which we will have no further use.



\subsection{A stable orbit morphism}\label{ section: a stable orbit morphism}

Let $Y=\{x \in X:~ x_1(e)=1\}$. For $y \in Y$, let $\cN_y:=(V_y,E_y,\cL_y,\cG_y,\rho_y)$ be  the rooted network induced by $y$  and $\cS=\{a,b\}$  as in \S \ref{ section: rooted networks from group actions }. Define $\cN^\phi_y=(V^\phi_y,E^\phi_y,\cL^\phi_y,\cG^\phi_y,\rho_y)$ by
\begin{itemize}
\item $V^\phi_y = \{g \in V_y=\Ft:~ y_1(g)=1\}$.
\item $\cL^\phi_y(g)= \big( (i_0,j_0), \ldots, (i_n,j_n)  \big)$ where for $0\le m \le n$, $y(ga^m)=(i_m,j_m)$ and $n\ge 0$ is the smallest number such that $y_1(ga^{n+1})=1$.
\item $\cG^\phi_y$ maps $E^\phi_y$ into $K \sqcup \{0\}$.
\item $E^\phi_y$ contains all edges of the form $(g,ga^n) \in V^\phi_y\times V^\phi_y$ where $g\in \Ft$  is any element with $y_1(g)=1$ and $n>0$ is the smallest number such that $y_1(ga^n)=1$. For any such edge define $\cG^\phi_y(g,ga^n):=0$.
\item  $E^\phi_y$ contains all edges of the form $\big(P_k(y,g),P_k(y,gb) \big)\in V^\phi_y\times V^\phi_y$ where $g \in \Ft$ is any element with $y_1(g)=k$. For any such edge define  $\cG^\phi_y \Big( \big(P_k(y,g),P_k(y,gb) \big)  \Big)=k$. For use later, define $\phi_y(g,gb):= \big(P_k(y,g),P_k(y,gb) \big)$.
\item The root $\rho_y$ is the identity element in $\Ft$.
\end{itemize}
 Warning: do not get $\cN^\phi_y$  confused with $\cN^\phi_x$  as defined in section \S \ref{ section: theorem one }.  They are completely different.  We will not need  the latter in this section.

\begin{lem}
For any $y\in Y$, $(V^\phi_y,E^\phi_y)$ is a tree.
\end{lem}

\begin{proof}
Let $g_1,g_2,\ldots$ be an  arbitrary ordering of the group $\Ft$. For each $n \ge 0$ let $\Gamma_n=(V,E_n)$ be the graph with vertex set $V=\Ft$ and edge set $E_n$  defined by
$$E_n:=\Big(E\cup  \big\{\phi_y(g_i,g_ib):~ i\le n \big\}\Big)- \big\{(g_j,g_jb):~ j \le n \big\}.$$  

{\bf Claim 1}. $\Gamma_n$ is a tree for all $n$.

Note that $\Gamma_0=(V,E)$ is the Cayley graph of $\Ft$. So it is a tree. For induction, assume that $\Gamma_n$ is a tree for some $n\ge 0$. So, the graph $\Gamma'_n$ obtained from $\Gamma_n$ by removing the edge $(g_{n+1},g_{n+1}b)$ has two components, each of which is a tree. The vertices $g_{n+1}$ and $g_{n+1}b$ are in different components of $\Gamma'_n$. Let $k=y_1(g_{n+1})$. Since $P_k(y,g_{n+1})=g_{n+1}a^m$ for some $m \in \Z$, it follows that $P_k(y,g_{n+1})$ and $g_{n+1}$ lie in the same component of $\Gamma'_n$. Similarly,  $P_k(y,g_{n+1}b)$ and $g_{n+1}b$ lie in the same component of $\Gamma'_n$. Thus, $\Gamma_{n+1}$, which is obtained from $\Gamma'_n$ by adding the edge $ \big(P_k(y,g_{n+1}), P_k(y,g_{n+1}b) \big)$ is a tree. This proves claim 1.

Let $\Gamma_\infty=(V,E_\infty)$ be the graph with vertex set $V=\Ft$ and edge set $E_\infty$ equal to the edge set $E$ minus the edges $\{(g,gb):~ g\in \Ft\}$ union the edges $\{\phi_y(g,gb):~ g\in \Ft  \}$.  It follows from claim 1 that  $\Gamma_\infty$ is a tree. Observe that if $g\in \Ft$ is such that $y_1(g)\ne 1$ then $g$ has degree 2 inside $\Gamma_\infty$. So $(V^\phi_y,E^\phi_y)$ is obtained from $\Gamma_\infty$ by removing all vertices of degree 2 and  gluing together the edges connecting such vertices. That is to say, if $y_1(g)=1$ and $n>0$ is the smallest number such that $y_1(ga^n)=1$ then we remove all the vertices of the form $ga^i$ for $0<i<n$ and all edges incident to such vertices and add in the edge $(g,ga^n)$. Clearly, this operation preserves simple connectivity. This proves the lemma.
\end{proof}

\begin{lem}
For any $y\in Y$, the rooted network $\cN^\phi_y$ is actionable. 
\end{lem}
\begin{proof}
This is an easy exercise left to the reader.
\end{proof}

Let $T=K \sqcup \{0\}$ and let $\F_T$ be the free group generated by $T$. Let $K_*$ be the set of all finite nonempty ordered lists of elements in $K\times K$. In other words, $K_* = \bigcup_{n=1}^\infty (K\times K)^n$. Define $\Omega: Y \to K_*^\Fk$ by $\Omega y(g) := \cL^\phi\big( A(\cN_y^\phi: \rho_y, g) \big)$.   The definition of $\Omega$ ensures that for any $y\in Y$ the network $\cN_{\Omega y}$ induced by $\Omega y$ and $T$ is isomorphic to $\cN^\phi_y$.  Warning:  this map is completely different from the map $\Omega$  defined in \S \ref{ section: theorem one }. We will not need the latter in this section. We will show that $\Omega $ is a stable orbit equivalence onto a Bernoulli shift over $\Fk$.

\begin{lem}\label{ lemma: stable morphism}
For any $y\in Y$, 
$$ \big\{ \Omega(g\cdot y):~ g\in \Ft \textrm{ such that }g\cdot y \in Y \big\}  =  \big\{ f\cdot (\Omega y):~ f\in \Fk \big\}.$$
\end{lem}
\begin{proof}
Let $y \in Y$ and let $\cN_y:=(V_y,E_y,\cL_y,\cG_y,\rho_y)$ be  the rooted network induced by $y$  and $\cS=\{a,b\}$. For $g\in \Ft$ such that $g\cdot y \in Y$, the rooted network $\cN_{g\cdot y}$ is isomorphic to $(V_y,E_y,\cL_y,\cG_y,g^{-1})$. This follows from lemma \ref{ lemma: root change }.

Let $\cN^\phi_y=(V^\phi_y,E^\phi_y,\cL^\phi_y,\cG^\phi_y,\rho_y)$. As mentioned previously, $\cN^\phi_y$ is isomorphic to $\cN_{\Omega y}$. By construction, it follows that $\cN_{\Omega(g\cdot y)}$ is isomorphic to $(V^\phi_y,E^\phi_y,\cL^\phi_y,\cG^\phi_y,g^{-1})$. Let $f \in \Fk$ be the unique element such that $A(\cN^\phi_y:\rho_y,f^{-1})=g^{-1}$. We know that such an element exists and is unique because  $(V^\phi_y,E^\phi_y)$ is a tree. Again, by lemma \ref{ lemma: root change }, $(V^\phi_y,E^\phi_y,\cL^\phi_y,\cG^\phi_y,g^{-1})$ is isomorphic to $\cN_{f \cdot \Omega(y)}$. This proves that $\Omega(g\cdot y) = f\cdot \Omega(y)$. Thus $\{ \Omega(g\cdot y):~ g\in \Ft, g\cdot y \in Y\}  \subset \{ f\cdot (\Omega y):~ f\in \Fk\}.$ The reverse inclusion is similar.
\end{proof}

\subsection{The inverse}
 In this section, we construct the inverse to $\Omega $.
\subsubsection{Pairings}
Given an element $\xi = \big( (i_0,j_0), \ldots, (i_n,j_n)  \big)\in K_*$, let $\len(\xi):=n$. In this paper $\N=\{0,1,2,\ldots\}$. For $z \in K_*^\Fk$, define 
$$\Fkz:= \big\{ (g,i) \in \Fk \times \N:~ \len[z(g)] \ge i \big\}.$$
Define a partial ordering on $ \Fkz$ by $(g,i) < (h,j)$ if either
\begin{enumerate}
\item there exists $n>0$ such that $ga^n=h$, or
\item $g=h$ and $i<j$.
\end{enumerate}
If there does not exist an $n$ such that $ga^n=h$ then $(g,i)$ and $(h,j)$ are not comparable.
 
For $(g,m) \in  \Fkz$, define $z(g,m):=(i_m,j_m)$ where $z(g)= \big( (i_0,j_0), \ldots, (i_n,j_n)  \big)$. Define $z_1(g,m):=i_m$ and $z_2(g,m):=j_m$.

For $z \in K_*^\Fk$, $g\in \Fk$ and $k\in K$ define $Q_k(z,g) \in  \Fkz$ as follows.
\begin{itemize}
\item If $k=1$, then $Q_k(z,g):=(g,0)$.
\item If $k\ne 1$ then let $Q_k(z,g)$ be the smallest element of $ \Fkz$ such that 
\begin{enumerate}
\item $(g,0) <  Q_k(z,g)$,
\item $z_1 \big(Q_k(z,g) \big)=k$ and
 \begin{eqnarray*}
 &&\Big|\big\{   (h,i) \in  \Fkz:~ z_1(h,i) = 1, (g,0) \le (h,i) \le Q_k(z,g)\big\}\Big|\\
  &=& \Big|\big\{   (h,i) \in  \Fkz:~ z_1(h,i) = k, (g,0) \le (h,i) \le Q_k(z,g)\big\}\Big|.
   \end{eqnarray*}
\end{enumerate}
\end{itemize}
A-priori, $Q_k(z,g)$ may not be well-defined. However, we have:
\begin{lem}
Let $Z$ be the set of all $z\in  K_*^\Fk$ such that
\begin{itemize}
\item for all $g\in \Fk$ and all $k\in K$, $Q_k(z,g)$ is well-defined;
 \item  $Q_k(z,\cdot)$ maps $\Fk$ bijectively onto the set $ \big\{ (g,i) \in \Fkz:~z_1(g,i)=k \big\}$. 
   \end{itemize}
  Then $\Omega (Y) \subset Z$. 
  \end{lem}
\begin{proof}
This is an easy exercise left to the reader.
\end{proof}

\subsubsection{The rooted network  of the inverse}

Recall that $T=K \sqcup \{0\}$. For ease of notation, we will write $s_k$ for the element of $\Fk$ corresponding to $k\in T$. For $z\in Z$, let $\cN^\psi_z:=( \Fkz,E^\psi_z,\cL^\psi_z,\cG^\psi_z,\rho_z)$ where
\begin{itemize}
\item $\cL^\psi_z(g,i)=z(g,i)$ for any $(g,i) \in  \Fkz$. 
\item $\cG^\psi_z$ maps $E^\psi_z$ into $\{a,b\}$.
\item $E^\psi_z$ contains all edges of the form $e=\big((g,i), (g,i+1)\big)$ for all $(g,i), (g,i+1) \in  \Fkz$. It also contains all edges of the form $e=\big( (g,n), (ga,0) \big)$ where $n=\len(z,g)$. For any such edge define $\cG^\psi_z(e):=a$.
\item  $E^\psi_z$ contains all edges of the form $e= \big( Q_k(z,f), Q_k( z, fs_k) \big)$ where $f\in\Fk$ is any element with $Q_k(z,f)=(g,i)$ for some $(g,i) \in \Fkz$ with $z_1(g,i)=k$. For each such edge define  $\cG^\psi_z(e)=b$.
\item The root $\rho_z=(e,0)$ where $e$ is the identity element in $\Fk$.
\end{itemize}

\begin{lem}
For any $z\in Z$, the rooted network $\cN^\psi_z$ is actionable. If $\Theta: Z \to (K\times K)^\Ft$ is defined by $\Theta z(g)=\cL^\psi \big( A(\cN^\psi_z: \rho_z, g) \big)$ then $\Theta$ is the inverse to $\Omega$. That is $\Omega \Theta (z)=z$ and $\Theta \Omega(y)=y$ for all $z\in Z$ and all $y\in Y$. 
 \end{lem}

\begin{proof}
This is an easy exercise left to the reader.
\end{proof}

Let $\mu$ be the probability measure on $Y$ defined by 
$$\mu(E) = \frac{\kappa^{\Ft/\langle b\rangle} \times \kappa^\Ft(E)}{\kappa^{\Ft/\langle b\rangle} \times \kappa^\Ft(Y)}$$
for any Borel $E \subset Y$.

\begin{cor}\label{ corollary: stable orbit equivalence }
$\Omega$ is a stable orbit-equivalence between the shift-action $\Ft \cc  \big((K\times K)^\Ft, \kappa^{\Ft/\langle b\rangle} \times \kappa^\Ft) \big)$ and the shift-action $\Fk \cc (K_*^\Fk, \Omega_*\mu)$.
\end{cor}
\begin{proof}
This follows from the lemma above and lemma \ref{ lemma: stable morphism}.
\end{proof}

\subsection{A measure space isomorphism}

\begin{prop}
 $\Omega_* \mu= \kappa_*^\Fk$ for some probability measure $\kappa_*$ on $K_*$.
 \end{prop}
 
 \begin{proof}
 For $y\in Y$, let $\cN_y=(V_y,E_y,\cL_y,\cG_y,\rho_y)$ be the rooted network induced by $y$ and $\cS=\{a,b\}$. Define $\cN^\phi_y=(V^\phi_y,E^\phi_y,\cL^\phi_V,\cG^\phi_y,\rho_y)$ as in \S \ref{ section: a stable orbit morphism}. For $g\in \Fk$, define  $\alpha_y(g)$ by   $\alpha_y(g):=A(\cN^\psi_y:\rho_y,g)$. So 
 $$\Omega y(g)=\Big(y \big(\alpha_y(g) \big), \ldots, y \big(\alpha_y(g)a^n \big)\Big)$$
 where $n \ge 0$ is the smallest number such that $y_1 \big(\alpha_y(g)a^{n+1} \big)=1$.  It suffices to show that if $y \in Y$ denotes a random variable with law $\mu$ then $ \big\{(\Omega y)(g) \big\}_{g\in \Fk}$ is a collection of  independent identically distributed (i.i.d.) random variables indexed by $\Fk$. 
 
Fix $k \in K$ and let $W_k^+=\{ g\in \Fk:~ |s_k^{-1}g| =|g|-1\}$ where $|\cdot |$ denotes the word metric. 
 
 Let $\tau(y) \in (K\times K)^\Ft$ be the function $\tau(y)(g):=y \big(P_k(y,e)bg \big)$. By construction, $y\mapsto \big[\Omega y(w)\big]_{w\in W_k^+}$ is  determined by $y\mapsto \big[\tau(y)(u)\big]_{u\in U}$ where $U:=\{g \in \Ft:~ \big|bg \big| = |g| +1 \}$.
 
 We claim that $y\mapsto \big[\tau(y)(u)\big]_{u\in U}$ is independent of $y\mapsto \big[y(a^n):~n\in\Z \big]$. To see this, observe that $y\mapsto P_k(y,e)$ is determined by $ \big[y(a^n):~n\in\Z \big]$. The sets $ \big\{a^n:~n\in\Z \big\}$ and  $\{P_k(y,e)bu:~u\in U\}$ are disjoint. There is a single coset $P_k(y,e)\langle b \rangle$  in the intersection of $ \big\{a^n\langle b \rangle:~n\in\Z \big\}$  and $\{P_k(y,e)bu\langle b \rangle:~u\in U\}$. These  facts imply that  the law of $y\mapsto \big[\tau(y)(u)\big]_{u\in U} $  conditioned on any arbitrary event $E$ in the $\sigma$-algebra induced by $y\mapsto \big[y(a^n):~n\in\Z \big]$ is the same as the law of $x\mapsto [x(u)]_{u\in U}$  where $x \in (K\times K)^\Ft$  is a random variable with law $\kappa^{\Ft/\langle b\rangle} \times \kappa^\Ft$ conditioned on  $x_1(e)=k$. In particular, $y\mapsto \big[\tau(y)(u)\big]_{u\in U} $ is independent of $y\mapsto \big[y(a^n):~n\in\Z \big]$ as claimed.

Since $y\mapsto \big[\Omega y(w)\big]_{w\in W_k^+}$ is  determined by $y\mapsto \big[\tau(y)(u)\big]_{u\in U}$ and $y\mapsto \big[\Omega y(s_0^n)\big]_{n\in \Z}$ is determined by  $y\mapsto \big[y(a^n):~n\in\Z \big]$, it follows that $y\mapsto \big[\Omega y(w)\big]_{w\in W_k^+}$ is independent of $y\mapsto \big[\Omega y(s_0^n)\big]_{n\in \Z}$. Let $W_k^-=\{ g\in \Fk:~ |s_kg| = |g| -1\}$. In a similar manner, it can be shown that $y\mapsto \big[\Omega y(w)\big]_{w\in W_k^-}$ is independent of $y\mapsto\big[\Omega y(s_0^n)\big]_{n\in \Z}$.

 It is an easy exercise to show that the variables $\{\Omega y(s_0^n)\}_{n\in \Z}$ are i.i.d.. Suppose, for induction, that $F \subset \Fk$ is a right-connected set  (as defined in \S \ref{ section: isomorphism }) such that $Fs_0 = F$ and $\{\Omega y(f)\}_{f \in F}$ is an i.i.d. collection. We claim that for any $g\in F$, any $k \in K$ and any $\epsilon \in \{-1,+1\}$, if $G=F \cup gs_k^\epsilon \langle s_0\rangle$ then $\{\Omega y(g)\}_{g \in G}$ is an i.i.d. collection. By induction, this will prove the proposition. 

To prove this claim, we may assume that $gs_k^\epsilon \notin F$ since otherwise $F=G$ and the claim is trivial. So $e \notin (gs_k^\epsilon)^{-1}F$. Since $g\in F$, $s_k^{-\epsilon} \in (gs_k^\epsilon)^{-1}F$. Since $F$ is right-connected, this implies $(gs_k^\epsilon)^{-1}F \subset W_k^{-\epsilon}$. We have already shown that $y\mapsto \big[\Omega y(w):~ w\in W_k^{-\epsilon}\big]$ is independent of $y\mapsto \big[\Omega y(s_0^n)\big]_{n\in \Z}$. Since $\Omega_*(\kappa^{\Ft/\langle b\rangle} \times \kappa^\Ft)$ is shift-invariant this implies $y\mapsto \big[\Omega y(f):~ f\in F\big]$ is independent of $y\mapsto \big[\Omega y(gs_k^\epsilon s_0^n)\big]_{n\in \Z}$. Since both collection of random variables are i.i.d. (by the induction hypothesis), this implies the claim and finishes the proposition.

\end{proof}

Theorem \ref{thm:2} follows immediately from the above and  corollary \ref{ corollary: stable orbit equivalence }.

\end{document}